\documentclass[11pt,reqno]{amsart}
\usepackage[T1]{fontenc}
\usepackage{lmodern,amsmath,amssymb,mathtools,microtype}
\usepackage[a4paper,margin=28mm]{geometry}
\usepackage[unicode,hidelinks]{hyperref}
\numberwithin{equation}{section}

\newtheorem{theorem}{Theorem}[section]
\newtheorem{proposition}[theorem]{Proposition}
\newtheorem{lemma}[theorem]{Lemma}
\newtheorem{corollary}[theorem]{Corollary}
\theoremstyle{remark}

\theoremstyle{plain}

\newcommand{\R}{\mathbb R}
\newcommand{\C}{\mathbb C}
\newcommand{\Q}{\mathbb Q}
\newcommand{\Z}{\mathbb Z}
\newcommand{\HH}{\mathbb H}
\newcommand{\DD}{\mathbb D}
\newcommand{\cA}{\mathcal A}
\newcommand{\cB}{\mathcal B}
\newcommand{\cK}{\mathcal K}
\newcommand{\cS}{\mathcal S}
\newcommand{\fl}{\mathfrak l}
\newcommand{\fk}{\mathfrak k}
\newcommand{\fp}{\mathfrak p}
\newcommand{\GL}{\operatorname{GL}}

\newcommand{\OO}{\operatorname{O}}
\newcommand{\SO}{\operatorname{SO}}
\newcommand{\UU}{\operatorname{U}}
\newcommand{\Sp}{\operatorname{Sp}}
\newcommand{\Mat}{\operatorname{Mat}}
\newcommand{\End}{\operatorname{End}}
\newcommand{\Hom}{\operatorname{Hom}}
\newcommand{\tr}{\operatorname{tr}}
\newcommand{\diag}{\operatorname{diag}}
\newcommand{\N}{\operatorname{N}}
\newcommand{\Pf}{\operatorname{Pf}}
\newcommand{\sgn}{\operatorname{sgn}}

\newcommand{\eps}{\varepsilon}
\newcommand{\absC}[1]{\lvert#1\rvert_{\C}}

\title[Uniqueness of twisted linear periods]
{Uniqueness of twisted linear periods\\over Archimedean division algebras}
\author[Y.F.~Li]{Yufeng Li}
\address{School of Mathematical Sciences, Peking University, No. 5 Yiheyuan Road, Beijing 100871, P.R. China}
\email{2401110016@stu.pku.edu.cn}

\author[J.Y.~Xiao]{Junyan Xiao}
\address{School of Mathematical Sciences, Peking University, No. 5 Yiheyuan Road, Beijing 100871, P.R. China}
\email{jyxiao25@stu.pku.edu.cn}

\date{October 14, 2026}

\begin{document}
\raggedbottom
\begin{abstract}
Let $\DD$ be $\R$, $\C$, or $\HH$, and let
$H=\GL_p(\DD)\times\GL_q(\DD)$ be the block diagonal subgroup of
$G=\GL_{p+q}(\DD)$. We prove that
$\dim\Hom_H(\pi,\eta)\leq1$ for every irreducible
Casselman--Wallach representation $\pi$ of $G$ and every smooth character
$\eta$ of $H$.
\end{abstract}
\maketitle
\section{Introduction}\label{sec:introduction}

Let $\DD\in\{\R,\C,\HH\}$, fix positive integers $p,q$, and put $n=p+q$.
We consider the pair
\[
 G=\GL_n(\DD),\qquad H=\GL_p(\DD)\times\GL_q(\DD),
\]
where $H$ is embedded block diagonally. A Casselman--Wallach
representation is a smooth admissible Fr\'echet representation of
moderate growth whose Harish--Chandra module has finite length;
see \cite{BK}. All representations are on complex vector spaces.
For such a representation $\pi$ on $E$ and a smooth character
$\eta:H\to\C^\times$, we write
\[
 \Hom_H(\pi,\eta)
 =\{\ell:E\to\C\text{ continuous and linear}:
       \ell(\pi(h)v)=\eta(h)\ell(v)\}.
\]
These functionals are the $\eta$-twisted linear periods of $\pi$.

\begin{theorem}\label{thm:main}
For every irreducible Casselman--Wallach representation $\pi$ of $G$
and every smooth character $\eta$ of $H$,
\begin{equation}\label{eq:main}
 \dim\Hom_H(\pi,\eta)\leq1.
\end{equation}
\end{theorem}

For the split pair over a local field, the untwisted multiplicity-one
theorem was proved in the non-Archimedean case by Jacquet and Rallis
\cite{JR}, and in characteristic zero by Aizenbud, Gourevitch, and Sayag
\cite{AGS09}. Twisted
generalized linear periods for adjacent block sizes were studied by Lu
\cite{Lu22}; the relation between generalized linear periods and Shalika
periods over division algebras was further studied in \cite{Lu26}.
For inner forms, further multiplicity-one results for untwisted linear
periods, including the near-balanced case for arbitrary irreducible
representations and arbitrary block sizes in the unitary case, were
obtained in \cite{ALMSY}.

For $\DD=\R,\C$ and $p=q$, Chen and Sun proved uniqueness under a
goodness condition on the character \cite{CS}. Their uniqueness theorem
for twisted Shalika periods, together with Geng's comparison of linear
and Shalika periods, also gives uniqueness for arbitrary characters on
generic representations \cite{CS,Geng}.
Compared with these results, Theorem~\ref{thm:main} places no condition
on $\eta$ and no genericity condition on $\pi$, allows arbitrary block
sizes $p,q$, and covers all three Archimedean division algebras.

\subsection{Reduction to invariant distributions}\label{subsec:reduction}
As usual, Theorem~\ref{thm:main} is deduced from a symmetry property of
equivariant distributions on $G$. Define a real-linear involutive
anti-automorphism of $\Mat_n(\DD)$ by
\begin{equation}\label{eq:s}
 s(M)=
 \begin{cases}
 M^t,&\DD=\R,\C,\\
 \overline M^{\,t},&\DD=\HH.
 \end{cases}
\end{equation}
The bar in the quaternionic case denotes quaternionic conjugation.
We also define an involutive automorphism $\alpha:G\longrightarrow G$ by
\[
 \alpha(g)=s(g)^{-1},\qquad g\in G.
\]
Let $L=H\times H$ act on matrices by
\begin{equation}\label{eq:action}
 (h_l,h_r)\cdot M=h_lMh_r^{-1}.
\end{equation}
Define the character $\chi_\eta:L\longrightarrow\C^\times$ by
\[
 \chi_\eta(h_l,h_r)=\eta(h_l)\eta(h_r)^{-1}.
\]

For a Nash manifold $Y$, let $\cS(Y)$ be its space of complex-valued
Schwartz functions and let $\cS'(Y)$ be the continuous complex-linear
dual \cite{AGSch}. We pair distributions with functions.
Suppose that a Nash group $\Gamma$ acts Nashly on $Y$. Its action on
$\cS(Y)$ is defined by
\[
 (\rho(g)f)(x)=f(g^{-1}\cdot x),
 \qquad g\in\Gamma,\quad f\in\cS(Y),\quad x\in Y.
\]
For a smooth character $\chi:\Gamma\longrightarrow\C^\times$, a
 distribution $T\in\cS'(Y)$ is $\chi$-equivariant if
\begin{equation}\label{eq:covariance}
 T(\rho(g)f)=\chi(g)T(f)
 \qquad(g\in\Gamma,\ f\in\cS(Y)).
\end{equation}
For a Nash involution $s:Y\longrightarrow Y$, put $s^*f=f\circ s$ and
$(s_*T)(f)=T(s^*f)$. We say that $T$ is $s$-invariant if
$s_*T=T$, and $s$-odd if $s_*T=-T$.

\begin{theorem}\label{thm:distributions}
For every smooth character $\eta$ of $H$, every
$\chi_\eta$-equivariant tempered distribution on $G$ is $s$-invariant.
\end{theorem}

\subsection{Organization}\label{subsec:organization}
The proof follows the implications
\[
 \text{Theorem~\ref{thm:main}}
 \ \Longleftarrow\
 \text{Theorem~\ref{thm:distributions}}
 \ \Longleftarrow\
 \text{Proposition~\ref{prop:matrix-distributions}}
 \ \Longleftarrow\
 \text{Theorem~\ref{thm:Gaussian}}
 +\text{Theorem~\ref{thm:quotient-symmetry}}.
\]
Sections~\ref{sec:Gaussian}--\ref{sec:matrix-model} reduce the
matrix-space distribution problem to symmetry in a polynomial quotient.
Sections~\ref{sec:identities}--\ref{sec:semi-invariants} establish this
symmetry using path identities, determinant and Pfaffian identities,
and quiver invariant theory.
Section~\ref{sec:conclusion} extends the distribution result to $G$
and deduces multiplicity one from the Gelfand--Kazhdan criterion.

\subsection{Conventions}
For a real vector space
$X$, the notation $\C[X]$ means the algebra of complex-valued
polynomials in real coordinates. Unless a complex vector space $V$
is explicitly regarded as a real vector space, $\C[V]$ denotes
$\operatorname{Sym}_{\C}(V^\vee)$, where $V^\vee$ is the complex-linear
dual. A based free module is a finite-rank free module equipped with a fixed
ordered basis. Direct sums carry the concatenated bases.

\section*{Acknowledgments}
This work is a continuation and application of a twisted version of Gaussian criterion developed in \cite{CCSP}.
The authors are grateful to their advisor, professor Jun Yu, for his guidance and valuable suggestions. 
The authors also acknowledge the assistance of GPT-6 Astra in locating references on computation of semi-invariants of quiver representations. 
This work was supported by NSFC Grants 12431001 and 11971036.

\section{Characters and the Gaussian criterion}\label{sec:Gaussian}

\subsection{Smooth characters}
Write $\HH=\C\oplus\C j$, where $j^2=-1$ and $jz=\bar z j$.  The
standard real-algebra embedding
\begin{equation}\label{eq:embedding}
 \iota:\Mat_r(\HH)\longrightarrow\Mat_{2r}(\C),
 \qquad
 Z+Wj\longmapsto
 \begin{pmatrix}Z&W\\-\overline W&\overline Z\end{pmatrix}
\end{equation}
defines the reduced norm
\[
 \N(A):=\operatorname{Nrd}(A)=\det_\C\iota(A).
\]
The function $\N$ is a nonnegative real polynomial on
$\Mat_r(\HH)$ and is positive precisely on $\GL_r(\HH)$.
Put $\absC z=z\bar z$ for $z\in\C$.

The smooth characters $\vartheta:\GL_r(\DD)\to\C^\times$ are
\begin{align}
 \vartheta(A)&=\sgn(\det A)^\eps|\det A|^z,
 &&\DD=\R,\quad \eps\in\{0,1\},\ z\in\C;
 \label{eq:char-R}\\
 \vartheta(A)&=(\det A)^m\absC{\det A}^{z},
 &&\DD=\C,\quad m\in\Z,\ z\in\C;
 \label{eq:char-C}\\
 \vartheta(A)&=\N(A)^z,
 &&\DD=\HH,\quad z\in\C.
 \label{eq:char-H}
\end{align}

We call a character $\eta_0$ of $H$ \emph{normalized} if it has the form
\begin{equation}\label{eq:normalized}
 \eta_0(A,D)=
 \begin{cases}
 \sgn(\det A)^\eps|\det A|^u|\det D|^v,
       &\DD=\R,\quad \eps\in\{0,1\},\quad u,v\in\C,\\
 (\det A)^m\absC{\det A}^u\absC{\det D}^v,
       &\DD=\C,\quad m\in\Z,\quad u,v\in\C,\\
 \N(A)^u\N(D)^v,
       &\DD=\HH,\quad u,v\in\C.
 \end{cases}
\end{equation}

By \eqref{eq:char-R}--\eqref{eq:char-H}, every smooth character $\eta$ of
$H$ can be written
\begin{equation}\label{eq:character-normalization}
 \eta=\eta_0\,\theta|_H,
\end{equation}
where $\eta_0$ is normalized and $\theta$ is an $s$-invariant
Nash character of $G$.

Theorem~\ref{thm:Gaussian} is a natural extension of the Gaussian
criterion in \cite[Section~3]{CCSP} from invariant tempered
distributions to tempered distributions equivariant under a
smooth character. Its proof follows the same Gaussian argument,
with the character incorporated into the averaging and
infinitesimal equivariance relations.

Let a real reductive group $L$ act linearly on a finite-dimensional
real vector space $X$. Fix a maximal compact subgroup $K$ of $L$ and a
Cartan decomposition $\fl=\fk\oplus\fp$. Suppose that $X$ has a
$K$-invariant positive definite inner product $\langle,\rangle$. Put
\begin{equation}\label{eq:Gaussian-ideal}
 \cA=\C[X],\qquad q_Z(x)=\langle Zx,x\rangle,
 \qquad J=(q_Z:Z\in\fp)\subset\cA.
\end{equation}
Let $s$ be an orthogonal linear involution of $X$ and let $\sigma$ be an
automorphism of $L$ preserving $K$ and $\fp$, such that
$s(gx)=\sigma(g)s(x)$ for all $g\in L, x\in X$. Fix a smooth character $\chi$ of $L$ satisfying
$\chi\circ\sigma=\chi$. For $k\in K$ and $Z\in\fp$, $K$-invariance gives
\[
 q_Z(kx)=q_{\operatorname{Ad}(k^{-1})Z}(x).
\]
Differentiating $s(gx)=\sigma(g)s(x)$ gives
$s(Zx)=d\sigma(Z)s(x)$. Since $s$ is an orthogonal involution,
\[
 q_Z(sx)=\langle s(Zsx),x\rangle=q_{d\sigma(Z)}(x).
\]
Hence $J$ is $K$- and $s^*$-stable. Denote the (complex-linear) $K$-action on $\cA$ and $\cA/J$ by $\rho$, and for a character $\psi$ of $K$, we set
\[
 (\cA/J)^{K,\psi}
 =\{[P]:[P(k\,\cdot)]=\rho(k^{-1})[P]=\psi(k)[P]\text{ for every }k\in K\}.
\]

\begin{theorem}\label{thm:Gaussian}
If $s^*$ acts identically on $(\cA/J)^{K,\chi^{-1}|_K}$, then every
$\chi$-equivariant tempered distribution on $X$ is $s$-invariant.
\end{theorem}
\begin{proof}
Since $\chi\circ\sigma=\chi$, the distribution $s_*T$ is
$\chi$-equivariant whenever $T$ is,
so the odd part of a $\chi$-equivariant distribution is again
$\chi$-equivariant. It therefore suffices to show that an $s$-odd
$\chi$-equivariant distribution $T$ vanishes. Set
$\gamma(x)=e^{-\|x\|^2/2}$. By Gaussian totality
\cite[Lemma~3.1]{CCSP}, the functions $P\gamma$, $P\in\cA$, are dense in
$\cS(X)$. We show that $T(P\gamma)=0$ by induction on $d=\deg P$.

Normalize Haar measure on $K$ to have total measure one. For $P\in\cA$, let
\[
 \mathsf p_\chi P=\int_K\chi(k)^{-1}\rho(k)P\,dk,
 \qquad
 P^-=\tfrac12(\mathsf p_\chi P-s^*\mathsf p_\chi P).
\]
The projection $\mathsf p_\chi$ preserves polynomial degree and commutes
with $s^*$. Its image satisfies
$\rho(k)\mathsf p_\chi P=\chi(k)\mathsf p_\chi P$.
Since $\gamma$ is fixed by $K$ and $s$, $K$-equivariance and $s$-oddness of $T$ give
\begin{equation}\label{eq:projected-Gaussian}
 T(P\gamma)=T(P^-\gamma).
\end{equation}
The class of $P^-$ lies in $(\cA/J)^{K,\chi^{-1}|_K}$ and satisfies
$s^*[P^-]=-[P^-]$. By hypothesis $s^*[P^-]=[P^-]$, so $P^-\in J$. As
$J$ is generated by homogeneous quadratics and $\deg P^-\leq d$, we may
write
\begin{equation}\label{eq:quadratic-decomposition}
 P^-=\sum_j q_{Z_j}Q_j,\qquad \deg Q_j\leq d-2.
\end{equation}
In particular $P^-=0$ if $d<2$, and then $T(P\gamma)=0$ by
\eqref{eq:projected-Gaussian}.

It remains to lower the degree of the terms $T(q_ZQ\gamma)$; this is
where the Gaussian enters. For $Z\in\fp$, put
\[
 D_ZQ(x)=\left.\frac{d}{dt}\right|_{t=0}Q(e^{-tZ}x),
 \qquad
 d\chi(Z)=\left.\frac{d}{dt}\right|_{t=0}\chi(e^{tZ}).
\]
Direct differentiation gives
\begin{equation}\label{eq:Gaussian-derivative}
 D_Z\gamma(x)
 =\left.\frac{d}{dt}\right|_{t=0}
   \exp\left(-\frac12\|e^{-tZ}x\|^2\right)
 =\langle Zx,x\rangle\gamma(x)
 =q_Z(x)\gamma(x).
\end{equation}
Differentiating \eqref{eq:covariance} along $e^{tZ}$ in the Schwartz
topology and applying the Leibniz rule to $D_Z(Q\gamma)$ together with
\eqref{eq:Gaussian-derivative} gives
\begin{equation}\label{eq:lowering}
 T(q_ZQ\gamma)=d\chi(Z)T(Q\gamma)-T((D_ZQ)\gamma).
\end{equation}
Both terms on the right involve polynomial factors of degree at most
$\deg Q\leq d-2$. Thus \eqref{eq:projected-Gaussian},
\eqref{eq:quadratic-decomposition}, \eqref{eq:lowering} and the
induction hypothesis give $T(P\gamma)=0$. Hence $T=0$.
\end{proof}

\section{The matrix model and the Gaussian quotient}\label{sec:matrix-model}

\subsection{Matrix coordinates}
Return to $L=H\times H$ acting on $X=\Mat_n(\DD)$.
For $r\geq1$, set
\[
 K_{\R}(r)=\OO_r(\R),\qquad
 K_{\C}(r)=\UU(r),\qquad
 K_{\HH}(r)=\Sp(r),
\]
where
\[
 \Sp(r)=\{g\in\GL_r(\HH):\overline g^{\,t}g=I_r\}.
\]
Then a maximal compact subgroup of $L$ is
\begin{equation}\label{eq:compact}
 K=(K_{\DD}(p)\times K_{\DD}(q))^2.
\end{equation}
Equip $X$ with the $K$-invariant real Frobenius inner product
\[
 \langle M,N\rangle=\operatorname{Re}\tr(M^*N).
\]
Here $M^*=M^t$ for $\DD=\R$, and
$M^*=\overline M^{\,t}$ for $\DD=\C,\HH$,
where the bar denotes complex or quaternionic conjugation,
respectively.

With respect to the standard Cartan decomposition
$\fl=\fk\oplus\fp$, every $Z=(Z_l,Z_r)\in\fp$ has
\[
 Z_l=\begin{pmatrix}A_l&0\\0&D_l\end{pmatrix},\qquad
 Z_r=\begin{pmatrix}A_r&0\\0&D_r\end{pmatrix},
\]
where
\[
 A_l^*=A_l,\qquad D_l^*=D_l,\qquad
 A_r^*=A_r,\qquad D_r^*=D_r.
\]
The involution $s$ of \eqref{eq:s} is orthogonal for this
inner product and normalizes the $L$-action through
\[
 \sigma(h_l,h_r)=(\alpha(h_r),\alpha(h_l)).
\]
This automorphism preserves $K$ and $\fp$. The character formulas
\eqref{eq:char-R}--\eqref{eq:char-H} give $\eta_0\circ s=\eta_0$
for every normalized character $\eta_0$, and hence
$\chi_{\eta_0}\circ\sigma=\chi_{\eta_0}$.

We keep $X=\Mat_n(\DD)$ as a real vector space in the Gaussian
criterion and use its algebraic complexification only to compute
$\cA/J$:
\[
 X_{\C}=X\otimes_{\R}\C,
 \qquad
 \cA=\C[X]\simeq\C[X_{\C}].
\]
The three algebraic models are as follows.

If $\DD=\R$, identify $X_{\C}$ with $\Mat_n(\C)$ and let
$M=(x_{ij})$ be the matrix coordinate. Then
\[
 \cA\simeq\C[M],\qquad Y:=M^t,
\]
where $\C[M]$ denotes the polynomial ring in the entries of $M$.

If $\DD=\C$, write a matrix in the underlying real vector space as
$A+\mathrm{i}B$, with $A,B\in\Mat_n(\R)$.  The linear change of
coordinates
\[
 M=A+\mathrm{i}B,\qquad
 Y=A^t-\mathrm{i}B^t
\]
gives
\[
 X_{\C}\simeq\Mat_n(\C)\oplus\Mat_n(\C),
 \qquad
 \cA\simeq\C[M,Y].
\]

If $\DD=\HH$, the complex-linear extension of the real-algebra
embedding \eqref{eq:embedding} is the isomorphism
\[
 \iota_{\C}:X_{\C}\xrightarrow{\sim}\Mat_{2n}(\C),
 \qquad A\otimes z\longmapsto z\,\iota(A).
\]
Put
\[
 \mathsf J_r=\begin{pmatrix}0&I_r\\-I_r&0\end{pmatrix}.
\]
Equip $\C^{2n}$ with the complex-bilinear symplectic form
\[
 \omega(u,v)=u^t\mathsf J_n v \qquad(u,v\in\C^{2n}),
\]
Direct calculation from \eqref{eq:embedding} gives
\[
 \iota(A^*)=\mathsf J_n^{-1}\iota(A)^t\mathsf J_n
 \qquad(A\in\Mat_n(\HH)).
\]

To make the $p+q$ block decomposition compatible with this model,
let $e_1,\ldots,e_{2n}$ be the original standard basis and use the
ordered basis
\[
 (e_1,\ldots,e_p,e_{n+1},\ldots,e_{n+p};
 e_{p+1},\ldots,e_n,e_{n+p+1},\ldots,e_{2n}).
\]
The corresponding change-of-basis matrix is
\[
 P=\begin{pmatrix}
 I_p&0&0&0\\
 0&0&I_q&0\\
 0&I_p&0&0\\
 0&0&0&I_q
 \end{pmatrix},
\]
with row block sizes $p,q,p,q$ and column block sizes $p,p,q,q$.
Set
\[
 \mathsf J=P^t\mathsf J_nP=\operatorname{diag}(\mathsf J_p,\mathsf J_q),
 \qquad \widetilde\iota(A)=P^{-1}\iota(A)P,
\]
and let $\widetilde\iota_{\C}$ be the complex-linear extension of
$\widetilde\iota$. In the reordered basis, the preceding adjoint
identity becomes
\[
 \widetilde\iota(A^*)
 =\mathsf J^{-1}\widetilde\iota(A)^t\mathsf J
 \qquad(A\in\Mat_n(\HH)).
\]
Moreover, for $B\in\Mat_p(\HH)$ and $C\in\Mat_q(\HH)$,
\[
 \widetilde\iota(\operatorname{diag}(B,C))
 =\operatorname{diag}(\iota(B),\iota(C)).
\]
Thus the two diagonal blocks have sizes $2p$ and $2q$.

For a linear map $B:V\to W$ between based symplectic spaces
$(V,\omega_V)$ and $(W,\omega_W)$, its symplectic adjoint
$B^\sharp:W\to V$ is defined by
\[
 \omega_W(Bv,w)=\omega_V(v,B^\sharp w)
 \qquad(v\in V,\ w\in W).
\]
If the form matrices in the chosen bases are
$\mathsf J_V,\mathsf J_W$, this identity gives
\begin{equation}\label{eq:sharp}
 B^\sharp=\mathsf J_V^{-1}B^t\mathsf J_W.
\end{equation}
In this notation, the transformed adjoint identity reads
\[
 \widetilde\iota(A^*)=\widetilde\iota(A)^\sharp\qquad(A\in\Mat_n(\HH)).
\]
Using $\widetilde\iota_{\C}$ to identify $X_{\C}$ with $\Mat_{2n}(\C)$,
let $M=(m_{ij})$ be the matrix coordinate and set
\[
 \cA\simeq\C[M],\qquad
 Y:=M^\sharp=\mathsf J^{-1}M^t\mathsf J.
\]
Thus $M\mapsto Y$ is the complex-linear extension of the quaternionic
adjoint.

In these models, $M$ and $Y$ are independent matrix coordinates only
when $\DD=\C$. For $\DD=\R$ or $\HH$, the entries of $M$ give
coordinates, and $Y$ is determined by $M$ via $Y=M^t$ or
$Y=M^\sharp$, respectively.

Let $\cK_{\DD}(r)$ be the complexification of $K_{\DD}(r)$.  Thus
\[
 \cK_{\R}(r)=\OO_r(\C),\qquad
 \cK_{\C}(r)=\GL_r(\C),\qquad
 \cK_{\HH}(r)=\Sp_{2r}(\C),
\]
and
\begin{equation}\label{eq:complexified-compact}
 \cK=(\cK_{\DD}(p)\times\cK_{\DD}(q))^2.
\end{equation}
The action of $K$ on $X$ extends algebraically to an action of
$\cK$ on $X_{\C}$. In the coordinates $M,Y$, this action is
\[
 (h_l,h_r)\cdot(M,Y)
 =\bigl(h_lMh_r^{-1},h_rYh_l^{-1}\bigr),
 \qquad (h_l,h_r)\in\cK.
\]
For $\DD=\R,\C$, use the standard ordered basis of $\C^n$ with
its $p+q$ block decomposition. For $\DD=\HH$, use the reordered
basis displayed above, with its $2p+2q$ block decomposition. Write
\begin{equation}\label{eq:blocks}
 M=\begin{pmatrix}a&b\\c&d\end{pmatrix},\qquad
 Y=\begin{pmatrix}y_a&y_c\\y_b&y_d\end{pmatrix}.
\end{equation}
For $\DD=\R$, the identity $Y=M^t$ gives
\[
 y_a=a^t,\qquad y_b=b^t,\qquad y_c=c^t,\qquad y_d=d^t.
\]
For $\DD=\C$, the blocks $y_a,y_b,y_c,y_d$ are independent of
$a,b,c,d$; their subscripts indicate the transposed positions of
the corresponding blocks of $M$.
For $\DD=\HH$, the identity $Y=M^\sharp$ gives
\[
 \begin{aligned}
 y_a&=\mathsf J_p^{-1}a^t\mathsf J_p=a^\sharp,
 &y_b&=\mathsf J_q^{-1}b^t\mathsf J_p=b^\sharp,\\
 y_c&=\mathsf J_p^{-1}c^t\mathsf J_q=c^\sharp,
 &y_d&=\mathsf J_q^{-1}d^t\mathsf J_q=d^\sharp.
 \end{aligned}
\]

\subsection{The Gaussian ideal and the symmetry target}

\begin{proposition}\label{prop:quadratic-ideal}
Under the identifications of $\cA=\C[X]$ with the matrix coordinate
rings above, the image of the Gaussian ideal \eqref{eq:Gaussian-ideal},
still denoted by $J$, is
\begin{equation}\label{eq:moment}
 J=(\text{entries of }ay_a+by_b,\ cy_c+dy_d,
                         \ y_aa+y_cc,\ y_bb+y_dd).
\end{equation}
\end{proposition}
\begin{proof}
For $Z=(Z_l,Z_r)\in\fp$,
\[
 q_Z(M)=\operatorname{Re}\tr(Z_lMM^*)
       -\operatorname{Re}\tr(Z_rM^*M).
\]
For $\DD=\R,\C$, both traces are real by self-adjointness of
$Z_l,Z_r$ and cyclicity of the trace, so $\operatorname{Re}$
can be omitted.
For $\DD=\HH$, the reordered embedding $\widetilde\iota$ satisfies
$\operatorname{Re}\tr(T)=\tfrac12\tr(\widetilde\iota(T))$ and
$\widetilde\iota(T^*)=\widetilde\iota(T)^\sharp$ for
$T\in\Mat_n(\HH)$, so we first express $q_Z$ in terms of
ordinary complex matrix traces.

We now extend $q_Z$, as a polynomial in the real coordinates of $X$,
to $X_{\C}$. Under the coordinate-ring identifications above, its
image is denoted by $q_Z(M,Y)$. Since $Y$ represents the
complex-linear extension of the adjoint, this image is given by
\[
 q_Z(M,Y)=
 \begin{cases}
 \tr(Z_lMY)-\tr(Z_rYM),&\DD=\R,\C,\\
 \tfrac12\bigl(\tr(Z_lMY)-\tr(Z_rYM)\bigr),&\DD=\HH.
 \end{cases}
\]
In the quaternionic case, $Z_l,Z_r$ here denote their images under
$\widetilde\iota$.
Moreover,
\[
 MY=
 \begin{pmatrix}
 ay_a+by_b&ay_c+by_d\\
 cy_a+dy_b&cy_c+dy_d
 \end{pmatrix},
\qquad
 YM=
 \begin{pmatrix}
 y_aa+y_cc&y_ab+y_cd\\
 y_ba+y_dc&y_bb+y_dd
 \end{pmatrix}.
\]
The diagonal blocks $A_l,D_l,A_r,D_r$ of $Z$ range over spaces whose
complex spans are the symmetric, arbitrary, or $\sharp$-symmetric
matrices of the relevant size, according as $\DD=\R,\C,\HH$; the
corresponding diagonal blocks of $MY$ and $YM$ lie in the same spaces,
on which the trace pairing is nondegenerate. Hence, as $Z_l$ and $Z_r$
vary, the forms $q_Z$ span the same space as the entries of the four
diagonal blocks above, giving \eqref{eq:moment}.
\end{proof}

The stability argument in Section~\ref{sec:Gaussian} now shows that
$J$ is $K$- and $s^*$-stable. Since $K$ is Zariski dense in $\cK$,
$J$ is also $\cK$-stable. Thus
\[
 R=\cA/J
\]
carries the induced $\cK$-action and involution $s^*$, where $s$ is
extended complex-linearly to $X_{\C}$.

\begin{theorem}\label{thm:quotient-symmetry}
For every normalized character $\eta_0$ in \eqref{eq:normalized},
the involution $s^*$ acts identically on
\[
 (\cA/J)^{K,\chi_{\eta_0}^{-1}|_K}.
\]
\end{theorem}

The proof is given in Section~\ref{subsec:quotient-symmetry-proof}.

\section{Two identities for trace-compatible modules}\label{sec:identities}

Throughout this section, $R$ denotes an arbitrary commutative
$\Q$-algebra. Let $\cB$ be a commutative unital $R$-algebra. Suppose that $\cB$ acts on based free $R$-modules
$V,W$ of the same rank $\ell$. For $c\in\cB$, write $c_V,c_W$ for
its actions. We call the two actions \emph{trace-compatible} if
\begin{equation}\label{eq:trace-compatible}
 \tr_V(c_V)=\tr_W(c_W)\qquad(c\in\cB).
\end{equation}
For $E=(e_{ij})\in\Mat_r(\cB)$, write
\[
 E_V\in\End_R(V^{\oplus r}),\qquad
 E_W\in\End_R(W^{\oplus r})
\]
for the induced endomorphisms. 
We use $E^{\mathsf T}=(e_{ji})$ for the transposition of $E$ over $\cB$.

\begin{lemma}\label{lem:trace-base}
Assume \eqref{eq:trace-compatible}. The endomorphisms $E_V,E_W$ have
the same characteristic polynomial, as do $E_V,(E^{\mathsf T})_V$.
The same conclusions hold after replacing
$\cB,R,V,W$ by
$\cB[[t]],R[[t]],V[[t]],W[[t]]$, respectively.
\end{lemma}
\begin{proof}
For $j\geq1$,
\[
 \tr(E_V^j)=
 \sum_{i_1,\ldots,i_j}
 \tr_V(e_{i_1i_2,V}e_{i_2i_3,V}\cdots e_{i_ji_1,V}).
\]
Each product is the action of an element of $\cB$, so the same sum
on $W$ agrees with it. Reversing the cyclic sequence of indices and
using commutativity gives the corresponding equality for
$(E^{\mathsf T})_V$. Newton's identities recover the characteristic
coefficients from these power traces over a $\Q$-algebra.

For $c(t)=\sum_{n\geq0}c_nt^n\in\cB[[t]]$, its actions on
$V[[t]]$ and $W[[t]]$ are $\sum_{n\geq0}c_{n,V}t^n$ and
$\sum_{n\geq0}c_{n,W}t^n$, whose traces agree coefficientwise by
\eqref{eq:trace-compatible}. So the extended actions are
trace-compatible, and the preceding argument over $R[[t]]$ applies.
\end{proof}

The first identity says that a block determinant built from two
$\cB$-linear maps does not change when the block array is transposed.

\begin{proposition}\label{prop:determinant}
Assume \eqref{eq:trace-compatible}, and let $A,B:V\to W$ be
$\cB$-linear. For $C=(c_{ij}),D=(d_{ij})\in\Mat_r(\cB)$, put
$F_{ij}=Ac_{ij,V}+Bd_{ij,V}$. Then
\begin{equation}\label{eq:block-determinant}
 \det_R(F_{ij})_{i,j=1}^r
 =\det_R(F_{ji})_{i,j=1}^r
 =\det_R(F_{ij}^{\,t})_{i,j=1}^r.
\end{equation}
\end{proposition}
\begin{proof}
The block matrix $(F_{ij})$ is $A^{\oplus r}C_V+B^{\oplus r}D_V$.
Extend $A,B$ coefficientwise to $V[[t]]\to W[[t]]$.
For $E\in\Mat_r(\cB[[t]])$, define
\[
 \Delta(E)=\det_{R[[t]]}(A^{\oplus r}+B^{\oplus r}E_V).
\]
If $S\in\GL_r(\cB[[t]])$, then, since $A$ and $B$ are $\cB$-linear,
\[
 A^{\oplus r}+B^{\oplus r}(SES^{-1})_V
 =S_W(A^{\oplus r}+B^{\oplus r}E_V)S_V^{-1}.
\]
Lemma~\ref{lem:trace-base} gives $\det S_W=\det S_V$, whence
\begin{equation}\label{eq:similarity}
 \Delta(SES^{-1})=\Delta(E).
\end{equation}

Put $\Lambda=\diag(1,\ldots,r)$ and set
\[
 C_t=I_r+t(C-I_r),\qquad
 D_t=\Lambda+t(D-\Lambda),\qquad E_t=D_tC_t^{-1}.
\]
The matrix $C_t$ is invertible over $\cB[[t]]$, and
$E_t=\Lambda+O(t)$. We diagonalize $E_t$ by successive
conjugations over $\cB[[t]]$. Suppose that a conjugate $E$
of $E_t$ has no off-diagonal terms of orders less than $n\geq1$,
and let $M$ be its coefficient of $t^n$. Conjugation in the form
$(I_r+t^nX)^{-1}E(I_r+t^nX)$ changes this coefficient to
$M+\Lambda X-X\Lambda$ and leaves the lower coefficients unchanged.
Taking $X_{ii}=0$ and $X_{ij}=-M_{ij}/(i-j)$ for $i\ne j$ eliminates its
off-diagonal entries. The resulting product of conjugating matrices
converges in the $t$-adic topology to an invertible matrix.
Thus $E_t$ is similar to a diagonal matrix, and hence to
$E_t^{\mathsf T}$.

Since
\[
 D_t^{\mathsf T}(C_t^{\mathsf T})^{-1}
 =C_t^{\mathsf T}E_t^{\mathsf T}(C_t^{\mathsf T})^{-1},
\]
equation~\eqref{eq:similarity} gives
$\Delta(D_t^{\mathsf T}(C_t^{\mathsf T})^{-1})=\Delta(E_t)$.
Factoring the two block matrices gives
\begin{align*}
 \det(A^{\oplus r}(C_t)_V+B^{\oplus r}(D_t)_V)
 &=\Delta(E_t)\det(C_t)_V,\\
 \det(A^{\oplus r}(C_t^{\mathsf T})_V
             +B^{\oplus r}(D_t^{\mathsf T})_V)
 &=\Delta\bigl(D_t^{\mathsf T}(C_t^{\mathsf T})^{-1}\bigr)
                         \det(C_t^{\mathsf T})_V.
\end{align*}
Lemma~\ref{lem:trace-base} gives
$\det(C_t)_V=\det(C_t^{\mathsf T})_V$, so the two left-hand
sides agree. Both belong to $R[t]$, and the inclusion
$R[t]\hookrightarrow R[[t]]$ is injective. Evaluating this
polynomial identity at $t=1$ proves the first equality in
\eqref{eq:block-determinant}. Transposing the entire block matrix
$(F_{ji})$ gives $(F_{ij}^{\,t})$, which proves the second equality.
\end{proof}

For a skew-symmetric matrix $M=(m_{ij})\in\Mat_{2r}(R)$, define
\[
 \Pf(M)=\frac{1}{2^r r!}
 \sum_{\sigma\in S_{2r}}\sgn(\sigma)
 \prod_{a=1}^{r}m_{\sigma(2a-1),\sigma(2a)}\in R,
\]
where $S_{2r}$ is the symmetric group.
This normalization gives
$\Pf\bigl(\begin{smallmatrix}0&a\\-a&0\end{smallmatrix}\bigr)=a$.
The Pfaffian is homogeneous of degree $r$ in the entries of $M$
and satisfies, for every $S\in\Mat_{2r}(R)$, the congruence identity
\begin{equation}\label{eq:Pf-congruence}
 \Pf(SMS^t)=\det(S)\Pf(M).
\end{equation}
The second identity exchanges the two diagonal coefficients of a block
Pfaffian.

\begin{proposition}\label{prop:Pfaffian}
Let a commutative $R$-algebra $\cB$ act trace-compatibly on two based
free rank-$r$ modules $W,V$, and suppose that its action matrices
are symmetric. Let $C_W,C_V$ be skew-symmetric endomorphisms
commuting with the respective $\cB$-actions, and let $F:V\to W$
be $\cB$-linear. Then, for all $h,k\in\cB$,
\begin{equation}\label{eq:Pfaffian-exchange}
 \Pf\begin{pmatrix}C_Wh_W&F\\-F^t&C_Vk_V\end{pmatrix}
 =\Pf\begin{pmatrix}C_Wk_W&F\\-F^t&C_Vh_V\end{pmatrix}.
\end{equation}
\end{proposition}
\begin{proof}
Conjugating by $\diag(r_W,r_V^{-1})$ with $r=(kh^{-1})^{1/2}$
exchanges $h$ and $k$ in the diagonal blocks, and trace compatibility
makes the determinant of this matrix equal to one. To make $h$ and $k$
invertible and the square root available, we replace them by $1+zh$
and $1+zk$ over formal power series in $z$ and compare top coefficients.

Extend $C_W,C_V,F$ coefficientwise to the corresponding
maps on $W[[z]]$ and $V[[z]]$.
In $\cB[[z]]$, put
\[
 h_z=1+zh,\qquad k_z=1+zk,\qquad
 r(z)=(k_zh_z^{-1})^{1/2},
\]
where the square root is the binomial series with constant term one.
Let $r_W,r_V$ denote the actions of $r(z)$ on $W[[z]]$
and $V[[z]]$, respectively.
The action matrices $r_W,r_V$ are symmetric. Trace compatibility
and the formal identity $\det r=\exp(\tr\log r)$ give
$\det r_W=\det r_V$.
Consequently
\[
 S=\diag(r_W,r_V^{-1})\quad\text{satisfies}\quad\det S=1.
\]
Let
\[
 M_z=\begin{pmatrix}
 C_Wh_{z,W}&zF\\-zF^t&C_Vk_{z,V}
 \end{pmatrix}
 =M_0+zM_1,
\]
where
\[
 M_0=\begin{pmatrix}C_W&0\\0&C_V\end{pmatrix},
 \qquad
 M_1=\begin{pmatrix}C_Wh_W&F\\-F^t&C_Vk_V\end{pmatrix}.
\]
The relation $r_WF=Fr_V$, the symmetry of $r_W,r_V$, and the
commutation hypotheses give
\[
 SM_zS^t=
 \begin{pmatrix}
 C_Wk_{z,W}&zF\\-zF^t&C_Vh_{z,V}
 \end{pmatrix}
 =M_0+zM_1',
 \qquad
 M_1'=\begin{pmatrix}C_Wk_W&F\\-F^t&C_Vh_V\end{pmatrix}.
\]
Since $\det S=1$, the congruence identity \eqref{eq:Pf-congruence} gives
\[
 \Pf(M_0+zM_1)=\Pf(M_0+zM_1').
\]
Both sides are polynomials in $z$ of degree at most $r$. By homogeneity
of the Pfaffian, their coefficients of $z^r$ are
$\Pf(M_1)$ and $\Pf(M_1')$, respectively.
Hence $\Pf(M_1)=\Pf(M_1')$, which is \eqref{eq:Pfaffian-exchange}.
\end{proof}

\section{The four-cycle quotient}\label{sec:fourcycle}

Return to the matrix coordinate ring $\cA=\C[X_{\C}]$ and the
Gaussian quotient $R=\cA/J$ of Section~\ref{sec:matrix-model}.

\subsection{The four-cycle model}
Let $V_0,V_1,V_2,V_3$ be complex vector spaces with fixed ordered
bases and dimensions $(p,p,q,q)$ for $\DD=\R,\C$, and
$(2p,2p,2q,2q)$ for $\DD=\HH$. For $B=\cA$ or $R$, put
\[
 V_{i,B}=V_i\otimes_{\C}B,\qquad
 U_B=\bigoplus_{i=0}^3V_{i,B}.
\]
All vertex indices are taken modulo four. The matrix coordinates
in \eqref{eq:blocks} define $\cA$-linear maps
\[
 M:V_{1,\cA}\oplus V_{3,\cA}\longrightarrow
       V_{0,\cA}\oplus V_{2,\cA},\qquad
 Y:V_{0,\cA}\oplus V_{2,\cA}\longrightarrow
       V_{1,\cA}\oplus V_{3,\cA},
\]
whose blocks have the following sources and targets:
\[
\begin{array}{llll}
 a:V_{1,\cA}\to V_{0,\cA},&b:V_{3,\cA}\to V_{0,\cA},&
 c:V_{1,\cA}\to V_{2,\cA},&d:V_{3,\cA}\to V_{2,\cA},\\
 y_a:V_{0,\cA}\to V_{1,\cA},&y_b:V_{0,\cA}\to V_{3,\cA},&
 y_c:V_{2,\cA}\to V_{1,\cA},&y_d:V_{2,\cA}\to V_{3,\cA}.
\end{array}
\]
Their entries are elements of $\cA$. For example, if
$a=(a_{\mu\nu})$, its action is
\[
 a(e_{1,\nu}\otimes r)
   =\sum_\mu e_{0,\mu}\otimes a_{\mu\nu}r
   \qquad(r\in\cA).
\]
Reduction modulo $J$ replaces each coefficient $f\in\cA$ by its
class $[f]\in R$ and gives the corresponding $R$-linear maps.
We use the same block letters after this reduction, specifying
the coefficient ring when needed. As in Section~\ref{sec:matrix-model},
$M,Y$ are independent only for $\DD=\C$; for $\DD=\R,\HH$,
they satisfy $Y=M^t$ and $Y=M^\sharp$, respectively.

Use the standard symmetric bilinear forms on the vertex spaces
for $\DD=\R,\C$. For $\DD=\HH$, use the standard symplectic
forms with matrix $\mathsf J_p$ on $V_0,V_1$ and $\mathsf J_q$
on $V_2,V_3$. Extend these forms $B$-bilinearly to $V_{i,B}$;
direct sums carry the orthogonal direct-sum forms.
For a map $F:V_{i,B}\to V_{j,B}$, write
$F^\diamond:V_{j,B}\to V_{i,B}$ for its adjoint: its matrix is
$F^t$ for $\DD=\R,\C$ and is given by \eqref{eq:sharp} for
$\DD=\HH$.

Set $\tau(i)=1-i\pmod4$, and let
\[
 j_i:V_i\longrightarrow V_{\tau(i)},\qquad
 e_{i,k}\longmapsto e_{\tau(i),k}.
\]
Then $j_{\tau(i)}=j_i^{-1}$. The paired spaces have identical form
matrices, so $j_i$ preserves their forms. Let $E$ be the automorphism
of $\bigoplus_iV_i$ whose restriction to $V_i$ is $j_i$, and use
the same notation after extension to $B$. In particular, $E^2=1$.

For $T\in\End_B(U_B)$, its adjoint
$T^\diamond:U_B\to U_B$ is taken with respect to the direct-sum
form. Explicitly,
\[
 T^\diamond=
 \begin{cases}
 T^t,&\DD=\R,\C,\\
 \mathsf J_{\mathrm{tot}}^{-1}T^t\mathsf J_{\mathrm{tot}},
      &\DD=\HH,
 \end{cases}
 \qquad
 \mathsf J_{\mathrm{tot}}
 =\diag(\mathsf J_p,\mathsf J_p,\mathsf J_q,\mathsf J_q).
\]

The involution $s^*f=f\circ s$ of $\cA$ preserves $J$ and induces
$s^*[f]=[s^*f]$ on $R$. On matrices over either ring, define $s^*$
entrywise:
\[
 s^*(f_{\mu\nu})=(s^*f_{\mu\nu}).
\]
That is, $s^*F:V_{i,B}\to V_{j,B}$ has the same source and target as $F$, and
$s^*(rF)=(s^*r)s^*F, r\in\cA$.
In the fixed coordinate bases of Section~\ref{sec:matrix-model},
the formulas $s^*M=M^\diamond$ and $s^*Y=Y^\diamond$ give
\[
\begin{aligned}
 s^*a&=a^\diamond,&s^*b&=c^\diamond,&
 s^*c&=b^\diamond,&s^*d&=d^\diamond,\\
 s^*y_a&=y_a^\diamond,&s^*y_b&=y_c^\diamond,&
 s^*y_c&=y_b^\diamond,&s^*y_d&=y_d^\diamond.
\end{aligned}
\]
These are equalities of coordinate matrices. When the sources and
targets are retained, the basis identifications must be inserted;
for example,
\[
 s^*a=j_1a^\diamond j_1,
 \qquad s^*b=j_1c^\diamond j_3.
\]

Define $\widetilde P,\widetilde Q\in\End_{\cA}(U_{\cA})$
by their nonzero blocks:
\begin{equation}\label{eq:PQ}
 \begin{array}{c|cccc}
 i&0&1&2&3\\\hline
 \widetilde P_i:V_{i,\cA}\to V_{i+1,\cA}&y_a&c&y_d&b\\
 \widetilde Q_i:V_{i+1,\cA}\to V_{i,\cA}&-a&y_c&-d&y_b.
 \end{array}
\end{equation}
Their commutator is block diagonal, with $i$-th block
$\widetilde P_{i-1}\widetilde Q_{i-1}
 -\widetilde Q_i\widetilde P_i$. Explicitly, over $\cA$,
\begin{equation}\label{eq:PQ-blocks}
 [\widetilde P,\widetilde Q]
 =\diag\bigl(ay_a+by_b,-(y_aa+y_cc),
                 cy_c+dy_d,-(y_bb+y_dd)\bigr).
\end{equation}
By \eqref{eq:moment}, the entries of this matrix generate $J$.
Let $P,Q\in\End_R(U_R)$ be the reductions of
$\widetilde P,\widetilde Q$, and denote their blocks by $P_i,Q_i$.
Then
\begin{equation}\label{eq:PQ-commute}
 PQ=QP\qquad\text{in }\End_R(U_R).
\end{equation}
The block formulas for $s^*$ above give, as endomorphisms of $U_R$,
\begin{equation}\label{eq:sPQ}
 P':=s^*P=EP^\diamond E^{-1},\qquad Q':=s^*Q=EQ^\diamond E^{-1}.
\end{equation}

A path $\gamma:i\to j$ means a composable word in the eight arrows
$\widetilde P_i,\widetilde Q_i$ of \eqref{eq:PQ}, with their signs
included and the rightmost arrow applied first. Its product
defines
\[
 F_{\gamma,\cA}\in
 \Hom_{\cA}(V_{i,\cA},V_{j,\cA}),\qquad
 F_{\gamma,R}\in\Hom_R(V_{i,R},V_{j,R})
\]
before and after reduction. A length-zero path gives the identity.
For a closed path at $i$, its trace is the polynomial
\[
 f_\gamma=\tr_{V_{i,\cA}}(F_{\gamma,\cA})\in\cA,
 \qquad
 [f_\gamma]=\tr_{V_{i,R}}(F_{\gamma,R})\in R.
\]
Since coefficient pullback respects matrix products and traces,
$s^*F_\gamma$ is obtained by replacing every arrow by its pulled-back
matrix, and
\[
 s^*f_\gamma=\tr(s^*F_{\gamma,\cA}),\qquad
 s^*[f_\gamma]=\tr(s^*F_{\gamma,R}).
\]

Suppose that $\gamma$ starts at $i$ and has $\alpha$ steps in
the $P$ direction and $\beta$ in the $Q$ direction. By
\eqref{eq:PQ-commute},
\[
 F_{\gamma,R}
 =(P^\alpha Q^\beta)|_{V_{i,R}}
 :V_{i,R}\longrightarrow V_{i+\alpha-\beta,R}.
\]
In particular, the path is closed exactly when
$\alpha-\beta\equiv0\pmod4$.

\subsection{Traces and paths}

\begin{lemma}\label{lem:closed-traces}
Let $\alpha,\beta$ be nonnegative integers with $\alpha+\beta>0$
and $\alpha-\beta\equiv0\pmod4$. Then
\[
 \tr_{V_{i,R}}\bigl((P^\alpha Q^\beta)|_{V_{i,R}}\bigr)
\]
is independent of $i$. Every positive-length closed-path trace in
$R$ is one of these traces.
\end{lemma}
\begin{proof}
Suppose first that $\alpha>0$. Put
\[
 A=P_{i-1}:V_{i-1,R}\longrightarrow V_{i,R},\qquad
 B=(P^{\alpha-1}Q^\beta)|_{V_{i,R}}
       :V_{i,R}\longrightarrow V_{i-1,R}.
\]
Then $AB=(P^\alpha Q^\beta)|_{V_{i,R}}$, whereas
$BA=(P^\alpha Q^\beta)|_{V_{i-1,R}}$ by
\eqref{eq:PQ-commute}. The rectangular trace identity
\[
 \tr_{V_{i,R}}(AB)=\tr_{V_{i-1,R}}(BA)
\]
therefore identifies the traces at consecutive vertices. Iteration
reaches all four vertices. If $\alpha=0$ and $\beta>0$, factor off
the leftmost $Q$-arrow instead; this compares the traces at $i$ and
$i+1$. The last assertion follows from the path formula above.
\end{proof}

\begin{corollary}\label{cor:trace-symmetry}
For every closed path $\gamma$, the involution on $R$ satisfies
$s^*[f_\gamma]=[f_\gamma]$.
\end{corollary}
\begin{proof}
For a positive-length closed path starting at $i$, write
$F_{\gamma,R}=(P^\alpha Q^\beta)|_{V_{i,R}}$.
By \eqref{eq:sPQ} and the rule reversing products under adjoints,
\[
 (P')^\alpha(Q')^\beta
 =E(Q^\beta P^\alpha)^\diamond E^{-1}.
\]
Consequently, as endomorphisms of $V_{i,R}$,
\[
 s^*F_{\gamma,R}
 =j_{\tau(i)}
   \bigl((Q^\beta P^\alpha)|_{V_{\tau(i),R}}\bigr)^\diamond j_i.
\]
Adjoints and change of basis preserve trace. Hence
\[
 s^*[f_\gamma]
 =\tr_{V_{\tau(i),R}}
       \bigl((Q^\beta P^\alpha)|_{V_{\tau(i),R}}\bigr)
 =[f_\gamma],
\]
where the last equality uses $PQ=QP$ and
Lemma~\ref{lem:closed-traces}. A length-zero path has constant
trace $\dim_{\C}V_i$, which is also fixed.
\end{proof}

The same calculation applies to a path from $V_{1,R}$ to $V_{0,R}$.
Here $\tau$ exchanges the two endpoints, so reversing the path and
then applying the identifications gives
$s^*F=j_1F^\diamond j_1$. Indeed, for
$F=(P^\alpha Q^\beta)|_{V_{1,R}}$ with
$\alpha-\beta\equiv-1\pmod4$, the adjoint block in
\eqref{eq:sPQ} is the adjoint of
$(Q^\beta P^\alpha)|_{V_{1,R}}$, which equals
$(P^\alpha Q^\beta)|_{V_{1,R}}$ in $R$.
In the paired bases this says
\begin{equation}\label{eq:sF}
 s^*F=F^\diamond\qquad\text{as coordinate matrices over }R.
\end{equation}
It holds also for linear combinations of these paths with
$s^*$-fixed coefficients, since $\diamond$ is linear over $R$.
For paths from $V_{0,R}$ to $V_{1,R}$, the corresponding map
identity is $s^*F=j_0F^\diamond j_0$, with the same coordinate
formula \eqref{eq:sF}.

\begin{lemma}\label{lem:two-paths}
Let $\cB=R[t,u,v]$ be the polynomial algebra in three commuting
indeterminates. For each vertex $i$, define an $R$-algebra action
$\rho_i:\cB\to\End_R(V_{i,R})$ by
\begin{equation}\label{eq:tuv}
 \rho_i(t)=(PQ)|_{V_{i,R}},\qquad
 \rho_i(u)=P^4|_{V_{i,R}},\qquad
 \rho_i(v)=Q^4|_{V_{i,R}}.
\end{equation}
The actions $\rho_0,\rho_1$ are trace-compatible. Every path from
$V_{1,R}$ to $V_{0,R}$ belongs to the $\cB$-span of
\begin{equation}\label{eq:two-paths}
 A=Q|_{V_{1,R}}=-a,\qquad
 B=P^3|_{V_{1,R}}=b\,y_d\,c,
\end{equation}
where $A,B\in\Hom_{\cB}(V_{1,R},V_{0,R})$ and the scalar
action on this Hom-module is
$c\cdot F=F\rho_1(c)=\rho_0(c)F$ for $c\in\cB$.
In the reverse direction, all paths belong to the $\cB$-span of
$P|_{V_{0,R}}$ and $Q^3|_{V_{0,R}}$ in
$\Hom_{\cB}(V_{0,R},V_{1,R})$.
\end{lemma}
\begin{proof}
The three operators in \eqref{eq:tuv} preserve each vertex and
commute by \eqref{eq:PQ-commute}, so they define the stated actions.
For a nonconstant monomial in $t,u,v$, the two traces agree by
Lemma~\ref{lem:closed-traces}. They also agree on constants in $R$,
since $V_{0,R}$ and $V_{1,R}$ have the same rank. By $R$-linearity,
\[
 \tr_{V_{0,R}}\rho_0(c)=\tr_{V_{1,R}}\rho_1(c)
 \qquad(c\in\cB).
\]
Commutation with $PQ,P^4,Q^4$ shows that every path intertwines
the actions at its endpoints; in particular, $A,B$ are
$\cB$-linear.

A path from $1$ to $0$ has matrix
$(P^\alpha Q^\beta)|_{V_{1,R}}$ with
$\alpha-\beta\equiv-1\pmod4$. Commutativity gives
\[
 (P^\alpha Q^\beta)|_{V_{1,R}}=
 \begin{cases}
 B\rho_1\bigl(t^\beta u^{(\alpha-\beta-3)/4}\bigr),
       &\alpha-\beta\geq3,\\
 A\rho_1\bigl(t^\alpha v^{(\beta-\alpha-1)/4}\bigr),
       &\alpha-\beta\leq-1.
 \end{cases}
\]
All exponents are nonnegative integers. The reverse direction
uses $\alpha-\beta\equiv1\pmod4$ and the generators
$P|_{V_{0,R}},Q^3|_{V_{0,R}}$ in the same way.
\end{proof}

\subsection{Adjoint decomposition of closed-path matrices}
Suppose that $\DD=\R$, so $Y=M^t$.
On $U_R$ introduce the symmetric $R$-bilinear form with matrix
\[
 S=\diag(I_p,-I_p,-I_q,I_q),
\]
and write $T^\dagger=S^{-1}T^tS$ for the adjoint of
$T\in\End_R(U_R)$. The arrow formulas give
\begin{equation}\label{eq:dagger}
 P^\dagger=Q.
\end{equation}
On endomorphisms of a single vertex, $\dagger$ is ordinary
transpose, since the corresponding block of $S$ is $I$ or $-I$.

In the polynomial algebra $\cB$ of Lemma~\ref{lem:two-paths}, put
\begin{equation}\label{eq:wOmega}
 w=u+v,\qquad\Omega=u-v.
\end{equation}
Write $t_i=\rho_i(t)$, $u_i=\rho_i(u)$, $v_i=\rho_i(v)$,
$w_i=\rho_i(w)$, and $\Omega_i=\rho_i(\Omega)$.
Equations~\eqref{eq:PQ-commute} and~\eqref{eq:dagger} imply
\begin{equation}\label{eq:adjoint-relations}
 t_i^t=t_i,\qquad u_i^t=v_i,\qquad u_iv_i=t_i^4
 \qquad\text{in }\End_R(V_{i,R}).
\end{equation}
The last identity is a relation among the images of the formal
variables under $\rho_i$. Thus $t_i,w_i$ are commuting symmetric
endomorphisms, and $\Omega_i$ is skew-symmetric and commutes with
both of them.

\begin{lemma}\label{lem:skew-paths}
In the $\C$-linear span of the closed-path matrices in
$\End_R(V_{i,R})$, the skew-symmetric elements are precisely
\begin{equation}\label{eq:skew-form}
 \Omega_i h(t_i,w_i),\qquad h\in\C[T,W].
\end{equation}
\end{lemma}
\begin{proof}
Every closed-path matrix is a monomial
$t_i^a u_i^b$ or $t_i^a v_i^b$, with $a,b\geq0$.
By \eqref{eq:adjoint-relations}, their transposes are interchanged.
For $b>0$,
\[
 t_i^a(u_i^b-v_i^b)
 =\Omega_i t_i^a
   \bigl(u_i^{b-1}+u_i^{b-2}v_i+\cdots+v_i^{b-1}\bigr).
\]
The factor in parentheses is symmetric in the commuting operators
$u_i,v_i$, so it is a polynomial in $w_i=u_i+v_i$ and
$u_iv_i=t_i^4$. For $b=0$, the skew part is zero. Applying
$F\mapsto(F-F^t)/2$ to a complex linear combination of closed
paths therefore gives \eqref{eq:skew-form}; every skew-symmetric
element equals its own skew part. Conversely,
$\Omega_i h(t_i,w_i)$ is skew-symmetric and expands into a complex
linear combination of closed paths.
\end{proof}

\section{Semi-invariants and their symmetry}\label{sec:semi-invariants}

In vertex order, write $g\in\cK$ from \eqref{eq:complexified-compact}
as $g=(g_0,g_1,g_2,g_3)$,
where $g_0,g_1\in\cK_{\DD}(p)$ and
$g_2,g_3\in\cK_{\DD}(q)$. The vertex spaces of
Section~\ref{sec:fourcycle} are distinct copies of the corresponding
standard complex representations, with the following roles:
\[
\begin{array}{c|cccc}
 \text{vertex space}&V_0&V_1&V_2&V_3\\\hline
 \text{matrix component}&\text{left }p&\text{right }p&
                          \text{left }q&\text{right }q\\
 \text{acting factor}&g_0&g_1&g_2&g_3
\end{array}
\]
Their dimensions are $(p,p,q,q)$ for $\DD=\R,\C$ and
$(2p,2p,2q,2q)$ for $\DD=\HH$. Setting
$h_l=\diag(g_0,g_2)$ and $h_r=\diag(g_1,g_3)$, the point action is
\[
 M(gx)=h_lM(x)h_r^{-1},\qquad
 Y(gx)=h_rY(x)h_l^{-1}\qquad(x\in X_{\C}).
\]
Consequently, for any path $\gamma:i\to j$, evaluation of the
polynomial path matrix $F_{\gamma,\cA}$ gives
\[
 F_\gamma(x)\in\Hom_{\C}(V_i,V_j),\qquad
 F_\gamma(gx)=g_jF_\gamma(x)g_i^{-1}.
\]
For a closed path at $i$, the polynomial
$f_\gamma(x)=\tr_{V_i}(F_\gamma(x))$ is therefore
$\cK$-invariant.

For a character $\psi$ of $\cK$, write
\[
 \cA^{\cK,\psi}
 =\{f\in\cA:f(gx)=\psi(g)f(x)\text{ for all }g\in\cK\},
\]
and use the same point-weight convention for $R=\cA/J$.
The group $K$ is Zariski dense in $\cK$, so these polynomial
weight spaces can equally be computed using $K$. In particular, the
averaging argument in the proof of Lemma~\ref{lem:quotient-surjection}
below gives, for every algebraic character $\psi$ of $\cK$,
\begin{equation}\label{eq:weight-identification}
 (\cA/J)^{K,\psi|_K}=R^{\cK,\psi}.
\end{equation}

\begin{lemma}\label{lem:quotient-surjection}
For every algebraic character $\psi$ of $\cK$, the quotient map
induces a surjection
\begin{equation}\label{eq:quotient-surjection}
 \cA^{\cK,\psi}\longrightarrow R^{\cK,\psi}.
\end{equation}
\end{lemma}
\begin{proof}
Let $[f]\in R^{\cK,\psi}$ and choose a polynomial lift $f$ of
degree at most $d$. Averaging in the finite-dimensional space
$\cA_{\leq d}$ gives
\[
 \widetilde f(x)=\int_K\psi(k)^{-1}f(kx)\,dk.
\]
It satisfies $\widetilde f(kx)=\psi(k)\widetilde f(x)$ and
$[\widetilde f]=[f]$. Zariski density gives
$\widetilde f\in\cA^{\cK,\psi}$.
\end{proof}

\subsection{Polynomial generators}
All paths below are taken over $\cA$. Auxiliary variables are independent commuting
indeterminates fixed by $\cK$ and $s^*$; an
\emph{auxiliary-variable coefficient} means a coefficient in these
variables. Extension of modules and algebra actions to auxiliary
polynomial rings is understood.

\begin{proposition}\label{prop:trace-generators}
In each of the three models, $\cA^{\cK}$ is generated by traces of
closed paths.
\end{proposition}
\begin{proof}
For $\DD=\C$, this is \cite[Theorem~1]{LBP}. For $\DD=\R,\HH$,
apply \cite[Theorems~1 and~2]{Zubkov}. These results give generators
by characteristic coefficients of closed paths in the corresponding
mixed-quiver model. Over $\C$, Newton's identities express these
coefficients in terms of traces of powers of the same closed paths.
\end{proof}

In the complex model, for $m\in\Z$, put
\begin{equation}\label{eq:lambda}
 \lambda_m(g)=\det(g_0)^m\det(g_1)^{-m}.
\end{equation}

\begin{proposition}\label{prop:det-generators}
Suppose $\DD=\C$ and $m>0$. The space
$\cA^{\cK,\lambda_m}$ is spanned by products of closed-path traces
and auxiliary-variable coefficients of
\begin{equation}\label{eq:path-determinants}
 \det(F_{ij})_{1\leq i,j\leq m},
 \qquad
 F_{ij}=\sum_\nu z_{ij\nu}F_{ij\nu},
\end{equation}
where $F_{ij\nu}:V_{1,\cA}\to V_{0,\cA}$ are paths and $z_{ij\nu}$
are independent auxiliary variables, one for each chosen path in
each block.
For $m<0$, interchange the endpoints and replace $m$ by $|m|$.
\end{proposition}
\begin{proof}
Under the $\cK$-action, every path $F_{ij\nu}:V_1\to V_0$ transforms as
\[
 F_{ij\nu}\longmapsto g_0F_{ij\nu}g_1^{-1}.
\]
Hence the block matrix $F=(F_{ij})_{1\leq i,j\leq m}$ transforms as
\[
 F\longmapsto (I_m\otimes g_0)F(I_m\otimes g_1^{-1}),
\]
and therefore
\[
 \det(F)\longmapsto
 \det(g_0)^m\det(g_1)^{-m}\det(F)
 =\lambda_m(g)\det(F).
\]
Thus every auxiliary-variable coefficient of \eqref{eq:path-determinants}
is a $\lambda_m$-semi-invariant, and multiplying by closed-path traces
preserves the same weight. The spanning assertion follows from the
contraction description and polarization of quiver semi-invariants
\cite[Theorems~2.1 and~2.3]{SVdB}.
For $m<0$, interchange the two endpoints and replace $m$ by $|m|$.
\end{proof}

In the real model define
\begin{equation}\label{eq:delta}
 \delta(g)=\det(g_0)\det(g_1),\qquad g\in\cK.
\end{equation}
Here $V_0$ and $V_1$ both have dimension $p$.

\begin{proposition}\label{prop:Pf-generators}
Suppose $\DD=\R$. The space $\cA^{\cK,\delta}$ is spanned by
products of closed-path traces and auxiliary-variable coefficients
of the Pfaffians
\begin{equation}\label{eq:path-Pfaffians}
 \Pf\begin{pmatrix}U_0&F\\-F^t&U_1\end{pmatrix},
\end{equation}
where
\[
 F=\sum_\nu z_\nu F_\nu,\qquad
 U_i=\sum_\mu x_{i\mu}U_{i\mu},
\]
the $F_\nu:V_1\to V_0$ are paths, the $U_{i\mu}$ are skew parts of
closed-path matrices at $V_i$, and the $x_{i\mu},z_\nu$ are commuting
auxiliary variables. The matrix in \eqref{eq:path-Pfaffians} has
size $2p$.
\end{proposition}
\begin{proof}
Put
\[
 B=\begin{pmatrix}U_0&F\\-F^t&U_1\end{pmatrix}.
\]
For $g\in\cK$, let $S=\diag(g_0,g_1)$. Since $g_i^{-1}=g_i^t$,
\[
 B\longmapsto SBS^t,
 \qquad
 \Pf(SBS^t)=\det(S)\Pf(B)=\delta(g)\Pf(B).
\]
Hence every auxiliary-variable coefficient of \eqref{eq:path-Pfaffians}
is a $\delta$-semi-invariant; multiplication by closed-path traces
preserves the same weight.

It remains to prove the spanning assertion. Put
\[
 \cK^+=\SO(V_0)\times\SO(V_1)\times\OO(V_2)\times\OO(V_3).
\]
By \cite[Theorem~1 and \S3.2]{Lopatin}, specialized to the present
four-cycle, $\cA^{\cK^+}$ is generated by closed-path traces and by
the auxiliary-variable coefficients of
\[
 \Pf(U_0),\qquad
 \Pf(U_1),\qquad
 \Pf\begin{pmatrix}U_0&F\\-F^t&U_1\end{pmatrix},
\]
where the first two types occur only when $p$ is even. Here $U_i$
and $F$ are as above. Indeed, the only special-orthogonal vertices
are $V_0$ and $V_1$, so these are precisely the Pfaffian generators
in the present four-cycle.

By \eqref{eq:Pf-congruence}, the $\cK/\cK^+$-weights of closed-path
traces and of the three Pfaffians above are, respectively,
\[
 1,\qquad (g\mapsto\det(g_0)),\qquad
 (g\mapsto\det(g_1)),\qquad \delta.
\]
Hence the $\delta$-isotypic component is generated over
$\cA^{\cK}$ by the coefficients of
\[
 \Pf\begin{pmatrix}U_0&F\\-F^t&U_1\end{pmatrix}
\]
together with products of a generator of weight $g\mapsto\det(g_0)$
and one of weight $g\mapsto\det(g_1)$; the latter are already included
by taking $F=0$.
Finally, $\cA^{\cK}$ is generated
by closed-path traces by Proposition~\ref{prop:trace-generators}.
\end{proof}

\subsection{The involution on the quotient}\label{subsec:quotient-symmetry-proof}

\begin{proof}[Proof of Theorem~\ref{thm:quotient-symmetry}]
For a normalized character $\eta_0$, the absolute-value and reduced-norm
factors in \eqref{eq:normalized} are trivial on $K$. In the vertex
notation of \eqref{eq:lambda} and \eqref{eq:delta}, the inverse point
weight $\chi_{\eta_0}^{-1}|_K$ is therefore the restriction of
$\lambda_{-m}$ for $\DD=\C$, of $\delta^\eps$ for $\DD=\R$,
and of the trivial character for $\DD=\HH$.
By \eqref{eq:weight-identification}, it is enough to show that $s^*$
acts identically on the following spaces:
\begin{align}
 &R^{\cK,\lambda_m}\quad(m\in\Z),&&\DD=\C;
 \label{eq:quotient-C}\\
 &R^{\cK}\ \text{and}\ R^{\cK,\delta},&&\DD=\R;
 \label{eq:quotient-R}\\
 &R^{\cK},&&\DD=\HH.
 \label{eq:quotient-H}
\end{align}

\emph{Invariants.}
Lemma~\ref{lem:quotient-surjection} and Proposition~\ref{prop:trace-generators}
show that $R^{\cK}$ is generated by the images of closed-path traces.
By Corollary~\ref{cor:trace-symmetry}, all these generators are fixed
by $s^*$. This proves \eqref{eq:quotient-H}, the invariant part of
\eqref{eq:quotient-R}, and the case $m=0$ of \eqref{eq:quotient-C}.

\emph{The complex case, $m\neq0$.}
Suppose first that $\DD=\C$ and $m>0$.
By Lemma~\ref{lem:quotient-surjection} and
Proposition~\ref{prop:det-generators}, and since trace factors are
fixed, it remains to treat the determinants
\eqref{eq:path-determinants}. Adjoin their auxiliary variables to $R$.
Lemma~\ref{lem:two-paths} gives
\[
 F_{ij}=Ac_{ij,V_{1,R}}+Bd_{ij,V_{1,R}},
 \qquad c_{ij},d_{ij}\in\cB=R[t,u,v],
\]
where $A,B$ are \eqref{eq:two-paths}, and the two endpoint actions of
$\cB$ are trace-compatible. By \eqref{eq:sF},
$s^*F_{ij}=F_{ij}^t$. Hence
\[
 s^*\det(F_{ij})=\det(F_{ij}^t)=\det(F_{ij})
\]
by Proposition~\ref{prop:determinant}, applied with $V=V_{1,R}$,
$W=V_{0,R}$ and $r=m$.
Every auxiliary coefficient is therefore fixed. For $m<0$, apply the
same proposition to the reversed endpoint modules $V=V_{0,R}$,
$W=V_{1,R}$, the generators $P|_{V_{0,R}},Q^3|_{V_{0,R}}$ from
Lemma~\ref{lem:two-paths}, and the analogue of \eqref{eq:sF} for paths
$V_0\to V_1$, which is proved in the same way.
This proves \eqref{eq:quotient-C}.

\emph{The real case, weight $\delta$.}
Use the notation $t,w,\Omega$ of
\eqref{eq:tuv} and \eqref{eq:wOmega}.
Adjoin the auxiliary variables in \eqref{eq:path-Pfaffians} to $R$,
and let $s^*$ fix them. By Lemma~\ref{lem:skew-paths}, applied
coefficientwise in these variables, every Pfaffian in
\eqref{eq:path-Pfaffians} has the form
\begin{equation}\label{eq:Phi}
 \Phi(h,k,F)=
 \Pf\begin{pmatrix}
 \Omega_0h(t_0,w_0)&F\\
 -F^t&\Omega_1k(t_1,w_1)
 \end{pmatrix},
\end{equation}
where $h,k$ are polynomials with coefficients in the auxiliary variables.
The algebra $\cB^+=R[T,W]$ acts on the two endpoint modules by
$T\mapsto t_i$, $W\mapsto w_i$. These actions are symmetric
and trace-compatible. Since $F$ is a linear combination of paths,
\eqref{eq:PQ-commute} gives $Ft_1=t_0F$ and $Fw_1=w_0F$.
The endomorphisms $\Omega_i$ are skew-symmetric and commute with
these actions.

Equations~\eqref{eq:sPQ} and \eqref{eq:adjoint-relations} give
\[
 s^*\bigl(\Omega_0h(t_0,w_0)\bigr)
   =-\Omega_1h(t_1,w_1),\qquad
 s^*\bigl(\Omega_1k(t_1,w_1)\bigr)
   =-\Omega_0k(t_0,w_0),
\]
and $s^*F=F^t$ by \eqref{eq:sF}. Consequently,
\[
 s^*\Phi(h,k,F)
 =\Pf\begin{pmatrix}
 -\Omega_1h(t_1,w_1)&F^t\\
 -F&-\Omega_0k(t_0,w_0)
 \end{pmatrix}.
\]
Put
\[
 \Pi=\begin{pmatrix}0&I_p\\ I_p&0\end{pmatrix},\qquad
 B'=\begin{pmatrix}
 \Omega_0k(t_0,w_0)&F\\
 -F^t&\Omega_1h(t_1,w_1)
 \end{pmatrix}.
\]
Then
\[
 -\Pi B'\Pi^t
 =\begin{pmatrix}
 -\Omega_1h(t_1,w_1)&F^t\\
 -F&-\Omega_0k(t_0,w_0)
 \end{pmatrix},
 \qquad \det\Pi=(-1)^p,
\]
and therefore
\[
 s^*\Phi(h,k,F)
 =\Pf(-\Pi B'\Pi^t)
 =(-1)^p\det(\Pi)\Pf(B')
 =\Phi(k,h,F).
\]
Proposition~\ref{prop:Pfaffian}, applied with $\cB=\cB^+$,
$W=V_{0,R}$, $V=V_{1,R}$,
$C_W=\Omega_0$, $C_V=\Omega_1$, now gives
\[
 \Phi(k,h,F)=\Phi(h,k,F).
\]
Thus every Pfaffian generator and every auxiliary coefficient is
fixed. Its trace factors are fixed by
Corollary~\ref{cor:trace-symmetry}. This proves
\eqref{eq:quotient-R} and completes the theorem.
\end{proof}

\section{From the quotient to multiplicity one}\label{sec:conclusion}

\subsection{Equivariant extension}\label{subsec:extension}

\begin{proposition}\label{prop:matrix-distributions}
For every normalized character $\eta_0$ in \eqref{eq:normalized},
every $\chi_{\eta_0}$-equivariant tempered distribution on
$\Mat_n(\DD)$ is $s$-invariant.
\end{proposition}
\begin{proof}
We apply Theorem~\ref{thm:Gaussian} with $X=\Mat_n(\DD)$,
$L=H\times H$, and $\chi=\chi_{\eta_0}$.
Section~\ref{sec:matrix-model} verifies the required compatibility
of the Frobenius inner product, $s$, and $\chi_{\eta_0}$, and
Proposition~\ref{prop:quadratic-ideal} identifies the Gaussian ideal
with \eqref{eq:moment}. Theorem~\ref{thm:quotient-symmetry} shows that
$s^*$ acts identically on $(\cA/J)^{K,\chi_{\eta_0}^{-1}|_K}$,
so the Gaussian criterion gives the assertion.
\end{proof}

\begin{lemma}\label{lem:weighted-extension}
Let a Lie group $L$ act linearly on a finite-dimensional real vector
space $X$. Let $q_0$ be a nonzero nonnegative polynomial on $X$ satisfying
\[
 q_0(gx)=a_0(g)q_0(x),\qquad g\in L,x\in X
\]
for a positive character $a_0$ of $L$, and put $U=\{q_0>0\}$.
Let $\chi:L\to\C^\times$ be a smooth character.
If $0\ne T\in\cS'(U)$ is $\chi$-equivariant, then there exist
$k\geq0$ and a nonzero $\chi a_0^k$-equivariant
$T'\in\cS'(X)$ such that $T'|_U=q_0^kT$.
If a linear involution $s$ of $X$ satisfies $q_0\circ s=q_0$
and makes $T$ odd, then $T'$ can be chosen odd.
\end{lemma}
\begin{proof}
By \cite[Corollary~5.4.4]{AGSch}, choose a tempered extension
$\widetilde T$ of $T$ to $X$. In the odd case, replace it by
$(\widetilde T-s_*\widetilde T)/2$.
The tempered distribution $\widetilde T$ has finite order $m$.
Since linear substitutions preserve order, each distribution
\[
 E_g(f)=\widetilde T(\rho(g)f)-\chi(g)\widetilde T(f),\qquad g\in L,
\]
has order at most $m$ and is supported on $\{q_0=0\}$.
Taking $k=m+1$, we have $q_0^kE_g=0$ for every $g\in L$,
since $q_0^kf$ vanishes with all derivatives of order at most $m$
on the support of $E_g$.
Since
\[
 q_0^k\rho(g)f=a_0(g)^k\rho(g)(q_0^kf),
\]
we obtain
\[
 (q_0^k\widetilde T)(\rho(g)f)
 =\chi(g)a_0(g)^k(q_0^k\widetilde T)(f).
\]
Take $T'=q_0^k\widetilde T$. Its restriction is $q_0^kT$, which
is nonzero because $q_0$ has no zeros on $U$.
Multiplication by $q_0^k$ preserves oddness.
\end{proof}

On the three matrix spaces define
\begin{equation}\label{eq:relative-polynomial}
 q_{\DD}(M)=
 \begin{cases}
 (\det M)^2,&\DD=\R,\\
 \absC{\det M},&\DD=\C,\\
 \N(M),&\DD=\HH.
 \end{cases}
\end{equation}
These are nonnegative real polynomials, $G=\{q_{\DD}>0\}$, and
$q_{\DD}\circ s=q_{\DD}$. Under the $L$-action,
\[
 q_{\DD}((h_l,h_r)\cdot M)
 =a_{\DD}(h_l,h_r)q_{\DD}(M),\qquad
 a_{\DD}(h_l,h_r)=q_{\DD}(h_l)q_{\DD}(h_r)^{-1}.
\]
Thus $a_\DD=\chi_{q_\DD|_H}$. For $\DD=\R,\C,\HH$ respectively,
\[
 q_\DD(\diag(A,D))=|\det A|^2|\det D|^2,\qquad
 \absC{\det A}\,\absC{\det D},\qquad
 \N(A)\N(D).
\]
Hence, for every $k\geq0$, the character $\eta_0\,q_\DD^k|_H$ is again
normalized whenever $\eta_0$ is.

\begin{proof}[Proof of Theorem~\ref{thm:distributions}]
Suppose first that $\eta_0$ is normalized. If a
$\chi_{\eta_0}$-equivariant distribution on $G$ is not $s$-invariant,
its odd part is nonzero. Lemma~\ref{lem:weighted-extension},
applied to $q_{\DD}$, produces a nonzero odd tempered distribution
on $\Mat_n(\DD)$ with character
$\chi_{\eta_0}a_{\DD}^k=\chi_{\eta_0q_\DD^k|_H}$, and
$\eta_0q_\DD^k|_H$ is normalized.
Proposition~\ref{prop:matrix-distributions} gives a contradiction,
proving the normalized case.

For arbitrary $\eta$, write $\eta=\eta_0\theta|_H$ as in
\eqref{eq:character-normalization}, and put $w=\theta^{-1}$ on $G$.
Both $w$ and $w^{-1}$ are Nash multipliers of $\cS(G)$.
For $T\in\cS'(G)^{L,\chi_\eta}$, since the relative character of $w$ is $\chi_{\theta|_H}^{-1}$,
$wT$ is $\chi_{\eta_0}$-equivariant.
The normalized case gives
$s_*(wT)=wT$. Since $w\circ s=w$, this is $w\,s_*T=wT$.
Multiplication by $w^{-1}$ yields $s_*T=T$.
\end{proof}

\subsection{Contragredients and periods}\label{subsec:periods}
Let $dM$ be Lebesgue density on $\Mat_n(\DD)$. A bi-invariant Haar
density on $G$ is
\begin{equation}\label{eq:Haar}
 dx=
 \begin{cases}
 |\det x|^{-n}\,dM,&\DD=\R,\\
 \absC{\det x}^{-n}\,dM,&\DD=\C,\\
 \N(x)^{-2n}\,dM,&\DD=\HH.
 \end{cases}
\end{equation}
The density \eqref{eq:Haar} is also preserved by $s$.

By multiplying the densities in (\ref{eq:Haar}), the equivariance of an ordinary distribution in \eqref{eq:covariance} is equivalent
to the generalized-function identities 
\begin{equation}\label{eq:generalized-equivariance}
 f(h_lxh_r^{-1})=\eta(h_l)\eta(h_r)^{-1}f(x),
\end{equation}
and, since $dx$ is $s$-invariant, $s$-invariance of a distribution is
equivalent to $f(s(x))=f(x)$.
We apply the generalized Gelfand--Kazhdan criterion
\cite[Theorem 2.3(a)]{SZGK} with $S_1=S_2=H$, $\chi_1=\eta$,
$\chi_2=\eta^{-1}$, and the anti-automorphism $x\mapsto s(x)$ of $G$.
Therefore
\begin{equation}\label{eq:GK}
 \dim\Hom_H(\pi,\eta)\,
 \dim\Hom_H(\pi^\vee,\eta^{-1})\leq1
\end{equation}
for every irreducible Casselman--Wallach representation $\pi$.

The two factors in \eqref{eq:GK} are compared through the involution
$\alpha$ of \eqref{eq:s}.

\begin{lemma}\label{lem:contragredient}
For every irreducible Casselman--Wallach representation $\pi$ of
$\GL_n(\DD)$,
\begin{equation}\label{eq:contragredient}
 \pi^\vee\simeq\pi\circ\alpha.
\end{equation}
\end{lemma}
\begin{proof}
For $\DD=\R,\C$, this is \cite[Theorem 2.4.2]{AGS}.
Suppose $\DD=\HH$. Let
\[
 C=\operatorname{Int}(jI_n)\circ\alpha.
\]
On the standard fundamental Cartan
$\{\diag(z_1,\ldots,z_n):z_i\in\C^\times\}$, one has
$C(h)=h^{-1}$; hence $C$ is a fundamental Chevalley involution.
By \cite[Corollary 1.4]{Adams}, $\pi\circ C\simeq\pi^\vee$.
Since $C$ differs from $\alpha$ by an inner automorphism,
$\pi\circ C\simeq\pi\circ\alpha$, proving \eqref{eq:contragredient}.
\end{proof}

\begin{proof}[Proof of Theorem~\ref{thm:main}]
The involution $\alpha$ preserves $H$, and
the character formulas \eqref{eq:char-R}--\eqref{eq:char-H} give
$\eta\circ\alpha=\eta^{-1}$. 
By Lemma~\ref{lem:contragredient},
\[
 \Hom_H(\pi^\vee,\eta^{-1})
 \simeq\Hom_H(\pi\circ\alpha,\eta^{-1})
 \simeq\Hom_H(\pi,\eta).
\]
This proves \eqref{eq:main}.
\end{proof}


\begin{thebibliography}{99}

\bibitem{Adams}
J.~Adams,
\emph{The real Chevalley involution},
Compos. Math. \textbf{150} (2014), no.~12, 2127--2142.

\bibitem{ALMSY}
U.~K.~Anandavardhanan, H.~Lu, N.~Matringe, V.~S\'echerre, and C.~Yang,
\emph{The sign of linear periods},
\href{https://arxiv.org/abs/2402.12106}{arXiv:2402.12106} (2024).

\bibitem{AGSch}
A.~Aizenbud and D.~Gourevitch,
\emph{Schwartz functions on Nash manifolds},
Int. Math. Res. Not. IMRN (2008), Art.~ID rnm155, 37~pp.

\bibitem{AGS}
A.~Aizenbud, D.~Gourevitch, and E.~Sayag,
\emph{$(\GL_{n+1}(F),\GL_n(F))$ is a Gelfand pair for any local field $F$},
Compos. Math. \textbf{144} (2008), no.~6, 1504--1524.

\bibitem{AGS09}
A.~Aizenbud, D.~Gourevitch, and E.~Sayag,
\emph{Generalized Harish--Chandra descent, Gelfand pairs, and an
Archimedean analog of Jacquet--Rallis's theorem},
Duke Math. J. \textbf{149} (2009), no.~3, 509--567.

\bibitem{BK}
J.~Bernstein and B.~Kr\"otz,
\emph{Smooth Fr\'echet globalizations of Harish-Chandra modules},
Israel J. Math. \textbf{199} (2014), no.~1, 45--111.

\bibitem{CCSP}
Y.~Li, J.~Xiao, and J.~Yu,
\emph{Regularity and the Gelfand Property for Complex Symmetric Pairs},
\href{https://arxiv.org/abs/2608.18928v2}{arXiv:2608.18928v2} (2026).

\bibitem{CS}
F.~Chen and B.~Sun,
\emph{Uniqueness of twisted linear periods and twisted Shalika periods},
Sci. China Math. \textbf{63} (2020), no.~1, 1--22.


\bibitem{Geng}
Z.~Geng,
\emph{On the existence of twisted Shalika periods: the Archimedean case},
\href{https://arxiv.org/abs/2501.11917}{arXiv:2501.11917} (2025).

\bibitem{JR}
H.~Jacquet and S.~Rallis,
\emph{Uniqueness of linear periods},
Compos. Math. \textbf{102} (1996), no.~1, 65--123.

\bibitem{Lu22}
H.~Lu,
\emph{The generalized linear periods},
J. Number Theory \textbf{233} (2022), 261--284.

\bibitem{Lu26}
H.~Lu,
\emph{The generalized linear period and the Shalika period over a division algebra},
J. Number Theory \textbf{279} (2026), 238--255.

\bibitem{LBP}
L.~Le Bruyn and C.~Procesi,
\emph{Semisimple representations of quivers},
Trans. Amer. Math. Soc. \textbf{317} (1990), no.~2, 585--598.

\bibitem{Lopatin}
A.~A.~Lopatin,
\emph{Invariants of quivers under the action of classical groups},
J. Algebra \textbf{321} (2009), no.~4, 1079--1106.

\bibitem{SVdB}
A.~Schofield and M.~Van den Bergh,
\emph{Semi-invariants of quivers for arbitrary dimension vectors},
Indag. Math. (N.S.) \textbf{12} (2001), no.~1, 125--138.

\bibitem{SZGK}
B.~Sun and C.-B.~Zhu,
\emph{A general form of Gelfand--Kazhdan criterion},
Manuscripta Math. \textbf{136} (2011), nos.~1--2, 185--197.


\bibitem{Zubkov}
A.~N.~Zubkov,
\emph{Invariants of mixed representations of quivers I},
J. Algebra Appl. \textbf{4} (2005), no.~3, 245--285.

\end{thebibliography}
\end{document}